\documentclass{amsart}
\usepackage[usenames,dvipsnames]{color}
\usepackage[colorlinks]{hyperref}
\AtBeginDocument{
	\hypersetup{
		colorlinks=true,
 		linkcolor=RoyalBlue,
  		citecolor=RedOrange,
	 }
}

\newcommand{\floor}[1]{\left \lfloor #1 \right \rfloor}
\newcommand{\cal}{\mathcal}
\newcommand{\eps}{\varepsilon}

\newcommand{\beq}[1]{\begin{equation}\label{#1}}
\newcommand{\eeq}{\end{equation}}

\newtheorem{theorem}{Theorem}

\theoremstyle{definition}

\theoremstyle{remark}
\newtheorem{remark}[theorem]{Remark}

\begin{document}

\title{Counting unlabeled interval graphs}

\author{Huseyin Acan}
\address{Department of Mathematics, Rutgers University, Piscataway, NJ 08854}
\email{huseyin.acan@rutgers.edu}
\thanks{The author is supported by National Science Foundation Fellowship (Award No.~1502650).}

\date{}

\begin{abstract}
We improve the bounds on the number of interval graphs on $n$ vertices. In particular, denoting by $I_n$ the quantity in question, we show that $\log I_n \sim n\log n$ as $n\to \infty$.
\end{abstract}

\maketitle

A simple undirected graph is an \emph{interval graph} if it is isomorphic to the intersection graph of a family of intervals on the real line.
Several characterizations of interval graphs are known; see~\cite[Chapter 3]{MM} for some of them. 
Linear time algorithms for recognizing interval graphs are given in~\cite{BL} and \cite{HMPV}.

In this paper, we are interested in counting interval graphs.
Let $I_n$ denote the number of unlabeled interval graphs on $n$ vertices.
(This is the sequence with id A005975 in the On--Line Encyclopedia of Integer Sequences~\cite{OEIS}.)
Initial values of this sequence are given by Hanlon~\cite{Hanlon}.
Answering a question posed by Hanlon~\cite{Hanlon}, Yang and Pippenger~\cite{YP} proved that the generating function
\[
I(x)=\sum_{n\ge 1} I_nx^n
\] 
diverges for any $x\not=0$ and they established the bounds
\beq{Pip}
\frac{n\log n}{3}+O(n) \le \log I_n \le n\log n+O(n).
\eeq

The upper bound in~\eqref{Pip} follows from $I_n\le (2n-1)!!=\prod_{j=1}^{n} (2j-1)$, where the right hand side is the number of matchings on $2n$ points. For the lower bound, the authors showed
\[
I_{3k} \ge k!/3^{3k}
\]
by finding an injection from $S_k$, the set of permutations of length $k$, to three-colored interval graphs of size $3k$.

Using an idea similar to the one in~\cite{YP}, we improve the lower bound in~\eqref{Pip} so that the main terms of the lower and upper bounds match.
In other words, we find the asymptotic value of $\log I_n$.

For a set $S$, we denote by ${S \choose k}$ the set of $k$-subsets of $S$.

\begin{theorem}\label{main}
As $n\to \infty$, we have 
\beq{goal}
\log I_n \ge n\log n- 2n\log\log n- O(n).
\eeq
\end{theorem}

\begin{proof}
We consider certain interval graphs on $n$ vertices with colored vertices. Let $k$ be a positive integer smaller than $n/2$ and $\eps$ a positive constant smaller than $1/2$.
For $1\le j\le k$, let $B_j$ and $R_j$ denote the intervals $[-j-\eps,-j+\eps]$ and $[j-\eps,j+\eps]$, respectively. 
These $2k$ pairwise-disjoint intervals will make up $2k$ vertices in the graphs we consider.
Now let $\cal W$ denote the set of $k^2$ closed intervals with one endpoint in $\{-k,\dots,-1\}$ and the other in $\{1,\dots,k\}$.
We color $B_1,\dots,B_k$ with blue, $R_1,\dots,R_k$ with red, and the $k^2$ intervals in $\cal W$ with white.

Together with $\cal S:=\{B_1,\dots,B_k,R_1,\dots,R_k\}$, each $\{J_1,\dots,J_{n-2k}\} \in {\cal W \choose n-2k}$ gives an $n$-vertex, three-colored interval graph.
For a given $\cal J=\{J_1,\dots,J_{n-2k}\}$, let $G_{\cal J}$ denote the colored interval graph whose vertices correspond to $n$ intervals in $\cal S \cup \cal J$, and let $\cal G$ denote the set of all $G_{\cal J}$.

Now let $G\in \cal G$. For a white vertex $w \in G$, the pair $(d_B(w),d_R(w))$, which represents the numbers of blue and red neighbors of $w$, uniquely determine the interval corresponding to $w$;
this is the interval $[-d_B(w),d_R(w)]$.
In other words, $\cal J$ can be recovered from $G_{\cal J}$ uniquely. Thus
\[
|\cal G| = {k^2 \choose n-2k}.
\]
Since there are at most $3^n$ ways to color the vertices of an interval graph with blue, red, and white, we have
\[
I_n\cdot 3^n \ge |\cal G| = {k^2 \choose n-2k}\ge \left(\frac{k^2}{n-2k}\right)^{n-2k} \ge \left(\frac{k^2}{n}\right)^{n}
\]
for any $k< n/2$. Setting $k=\floor{n/\log n}$ and taking the logarithms, we get
\[
\log I_n \ge n\log(k^2/n)-O(n)= n\log n-2n\log\log n -O(n). \qedhere
\]
\end{proof}

\begin{remark}
Yang and Pippenger~\cite{YP} posed the question whether
\[
\log I_n=Cn\log n+O(n)
\]
for some $C$ or not. According to Theorem~\ref{main}, this boils down to getting rid of the $2n\log\log n$ term in~\eqref{goal}. 
Such a result would imply that the exponential generating function
\[
J(x)=\sum_{n\ge 1} I_n\frac{x^n}{n!}
\]
has a finite radius of convergence. (As noted in~\cite{YP}, the bound $I_n\le (2n-1)!!$ implies that the radius of convergence of $J(x)$ is at least $1/2$.)
Of course, a strong result would be finding $I_n$ asymptotically.
\end{remark}

\bibliographystyle{amsplain}

\end{document}